\documentclass[fleqn,reqno,11pt,a4paper,final]{amsart}
\usepackage[a4paper,left=35mm,right=35mm,top=35mm,bottom=35mm,marginpar=25mm]{geometry} 
\usepackage{amsmath}
\usepackage{amssymb}
\usepackage{amsthm}
\usepackage{amscd}
\usepackage[ansinew]{inputenc}
\usepackage{cite}
\usepackage{bbm}
\usepackage{xcolor}
\usepackage[english=american]{csquotes}
\usepackage[final]{graphicx}
\usepackage{hyperref}
\usepackage{calc}
\usepackage{bm}
\usepackage{enumerate}

\graphicspath{{Pictures/}}

\numberwithin{equation}{section}

\newtheoremstyle{thmlemcorr}{10pt}{10pt}{\itshape}{}{\bfseries}{.}{10pt}{{\thmname{#1}\thmnumber{ #2}\thmnote{ (#3)}}}
\newtheoremstyle{thmlemcorr*}{10pt}{10pt}{\itshape}{}{\bfseries}{.}\newline{{\thmname{#1}\thmnumber{ #2}\thmnote{ (#3)}}}
\newtheoremstyle{defi}{10pt}{10pt}{\itshape}{}{\bfseries}{.}{10pt}{{\thmname{#1}\thmnumber{ #2}\thmnote{ (#3)}}}
\newtheoremstyle{remexample}{10pt}{10pt}{}{}{\bfseries}{.}{10pt}{{\thmname{#1}\thmnumber{ #2}\thmnote{ (#3)}}}
\newtheoremstyle{ass}{10pt}{10pt}{}{}{\bfseries}{.}{10pt}{{\thmname{#1}\thmnumber{ A#2}\thmnote{ (#3)}}}

\theoremstyle{thmlemcorr}
\newtheorem{theorem}{Theorem}
\numberwithin{theorem}{section}
\newtheorem{lemma}[theorem]{Lemma}

\newtheorem{question}[theorem]{Question}

\theoremstyle{thmlemcorr*}
\newtheorem{theorem*}{Theorem}
\newtheorem{lemma*}[theorem]{Lemma}
\newtheorem{corollary*}[theorem]{Corollary}
\newtheorem{proposition*}[theorem]{Proposition}
\newtheorem{problem*}[theorem]{Problem}
\newtheorem{conjecture*}[theorem]{Conjecture}
\newtheorem{question*}[theorem]{Question}

\theoremstyle{defi}
\newtheorem{definition}[theorem]{Definition}

\theoremstyle{remexample}
\newtheorem{remark}[theorem]{Remark}

\theoremstyle{ass}

\newcommand{\Crm}{\mathrm{C}}

\newcommand{\Lrm}{\mathrm{L}}

\newcommand{\Wrm}{\mathrm{W}}

\newcommand{\Acal}{\mathcal{A}}

\newcommand{\Dcal}{\mathcal{D}}
\newcommand{\Ecal}{\mathcal{E}}
\newcommand{\Fcal}{\mathcal{F}}

\newcommand{\Hcal}{\mathcal{H}}

\newcommand{\Lcal}{\mathcal{L}}
\newcommand{\Mcal}{\mathcal{M}}

\newcommand{\Ebf}{\mathbf{E}}

\newcommand{\Ybf}{\mathbf{Y}}

\newcommand{\Abb}{\mathbb{A}}
\renewcommand{\Bbb}{\mathbb{B}}

\newcommand{\Sbb}{\mathbb{S}}

\DeclareMathOperator{\id}{id}

\DeclareMathOperator*{\wslim}{w*-lim}

\DeclareMathOperator{\diverg}{div}
\DeclareMathOperator{\curl}{curl}

\DeclareMathOperator{\spn}{span}

\newcommand{\setb}[2]{\bigl\{\, #1 \ \ \textup{\textbf{:}}\ \ #2 \,\bigr\}}

\newcommand{\setBB}[2]{\biggl\{\, #1 \ \ \textup{\textbf{:}}\ \ #2 \,\biggr\}}

\newcommand{\abs}[1]{|#1|}

\newcommand{\dpr}[1]{\langle #1 \rangle}	

\newcommand{\dprn}[1]{\langle #1 \rangle}
\newcommand{\dprb}[1]{\bigl\langle #1 \bigr\rangle}

\newcommand{\ddprb}[1]{\bigl\langle\hspace{-2.5pt}\bigl\langle #1 \bigr\rangle\hspace{-2.5pt}\bigr\rangle}

\newcommand{\cl}[1]{\overline{#1}}
\newcommand{\di}{\mathrm{d}}
\newcommand{\dd}{\;\mathrm{d}}

\newcommand{\N}{\mathbb{N}}
\newcommand{\R}{\mathbb{R}}

\newcommand{\loc}{\mathrm{loc}}
\newcommand{\sym}{\mathrm{sym}}

\newcommand{\toweakstar}{\overset{*}\rightharpoonup}

\newcommand{\sbullet}{\begin{picture}(1,1)(-0.5,-2)\circle*{2}\end{picture}}
\newcommand{\frarg}{\,\sbullet\,}
\newcommand{\BV}{\mathrm{BV}}
\newcommand{\BD}{\mathrm{BD}}

\newcommand{\BDY}{\mathbf{BDY}}

\newcommand{\toY}{\overset{\Ybf}{\to}}

\newcommand{\Rdds}{(\R^{d}\otimes\R^d)_\sym}

\def\Xint#1{\mathchoice 
{\XXint\displaystyle\textstyle{#1}}%
{\XXint\textstyle\scriptstyle{#1}}%
{\XXint\scriptstyle\scriptscriptstyle{#1}}%
{\XXint\scriptscriptstyle\scriptscriptstyle{#1}}%
\!\int} 
\def\XXint#1#2#3{{\setbox0=\hbox{$#1{#2#3}{\int}$} 
\vcenter{\hbox{$#2#3$}}\kern-.5\wd0}} 
\def\dashint{\,\Xint-}

\newcommand{\restrict}{\begin{picture}(10,8)\put(2,0){\line(0,1){7}}\put(1.8,0){\line(1,0){7}}\end{picture}}

\renewcommand{\phi}{\varphi}

\title[Structure of PDE-constrained measures]{On the structure of measures constrained by linear PDEs}

\author[G.~De Philippis]{Guido De Philippis}
\address{\textit{G.~De Philippis:} SISSA, Via Bonomea 265, 34136 Trieste, Italy.}
\email{guido.dephilippis@sissa.it}

\author[F.~Rindler]{Filip Rindler}
\address{\textit{F.~Rindler:} Mathematics Institute, University of Warwick, Coventry CV4 7AL, UK.}
\email{F.Rindler@warwick.ac.uk}

\begin{document}

\begin{abstract}
The aim of this note is to  present  some  recent results on  the structure of the  singular part of  measures satisfying a PDE constraint and to describe some  applications.
\end{abstract}

\maketitle

\section{Introduction}


We describe recent advances obtained by the authors and collaborators concerning the structure of singularities in measures satisfying a linear PDE constraint. Besides its own  theoretical interest, understanding the structure of singularities of PDE-constrained measures turns out to have several (sometimes surprising) applications in the calculus of variations, geometric measure theory, and metric  geometry.

Let \(\Acal\) be a    $k$'th-order linear constant-coefficient differential operator acting  on \(\R^N\)-valued functions, i.e.\
\[
  \Acal u := \sum_{|\alpha|\le k}A_{\alpha} \partial^\alpha u, \qquad
  u\in \Crm^\infty(\Omega;\R^N),
\]
where  $A_{\alpha}\in \R^n\otimes \R^N$ are linear maps from $\R^N$ to $\R^n$  and   $\partial^\alpha=\partial_1^{\alpha_1}\ldots\partial_d^{\alpha_d}$ for every multindex \(\alpha=(\alpha_1,\dots,\alpha_d)\in \mathbb (\N \cup \{0\})^d\). 

The starting point of the investigation is the following:

\begin{question}\label{q:Afree}
Let $\mu\in \Mcal (\Omega,\R^N) $ be  an \(\R^N\)-valued Radon measure  on an open set $\Omega \subset \R^d$  and let \(\mu\) be \(\Acal\)-free, i.e.\ $\mu$ solves the system of linear PDEs
\begin{equation} \label{eq:muPDE}
  \Acal \mu := \sum_{|\alpha|\le k}A_{\alpha} \partial^\alpha \mu = 0  \qquad \text{in the sense of distributions.}
\end{equation}
What can be said about the singular part\footnote{If not specified,   the terms ``singular'' and ``absolutely continuous'' always refer to the Lebesgue measure. We also recall that, thanks to the Radon-Nikodym theorem, a vector-valued measure $\mu$ can be written as 
\[
\mu=\frac{\di \mu } {\di |\mu|} \di |\mu|=g \Lcal^d+\frac{\di \mu } {\di |\mu|} \di |\mu|^s
\]
where \(|\mu|\) is the total variation measure, \(g\in \Lrm_\loc^1(\R^d)\) and \(\Lcal^d\) is the Lebesgue measure. 
} of \(\mu\)? 
\end{question}

In answering  the above question a prominent role is played by the \emph{wave cone} associated with the differential operator \(\Acal\):
\[
  \Lambda_{\Acal}:=\bigcup_{|\xi|=1} \ker \mathbb A^k(\xi) \subset \R^N \qquad\textrm{with}\qquad \mathbb A^k(\xi)= (2\pi {\rm i} )^{k} \sum_{|\alpha|=k}A_{\alpha}\xi^{\alpha},
\]
where we have set $\xi^\alpha := \xi_1^{\alpha_1} \cdots \xi_d^{\alpha_d}$. 

Roughly speaking,  \(\Lambda_{\Acal}\)   contains  all the amplitudes along which the system \eqref{eq:muPDE} is  \emph{not elliptic}. Indeed if we assume that  \(\Acal\) is homogeneous, \(\Acal=\sum_{|\alpha|=k}A_\alpha \partial^\alpha\), then   it is immediate to see that \(\lambda\in\R^N\) belongs to \( \Lambda_{\Acal}\) if and only if there exists a non-zero \(\xi\in \R^d\setminus\{0\}\) such that  \(\lambda h(x\cdot \xi)\) is  \(\Acal\)-free for all smooth functions \(h \colon \R\to \R\). In other words, ``one-dimensional''  oscillations and concentrations are  possible only if the amplitude (direction) belongs to the wave cone. For this reason  the  wave cone  plays a crucial role in the compensated compactness theory for sequences of $\Acal$-free maps, see~\cite{Murat78, Murat79, Tartar79, Tartar83, DiPerna85}, and in convex integration theory, see for instance~\cite{ChiodaroliDeLellisKreml15,ChiodaroliFeireislKremWiedemann17,DeLellisSzekelyhidi09,DeLellisSzekelyhidi12,DeLellisSzekelyhidi13,SzekelyhidiWiedemann12,Isett16} and the references cited therein. However, all these references are concerned with oscillations only, not with concentrations.

Since the singular part of a measure can be thought of as containing  ``condensed'' concentrations, it is quite natural to conjecture that \(|\mu|^s\)-almost everywhere the polar vector \(\frac{\di \mu}{\di |\mu|}\) belongs to \(\Lambda_{\Acal}\). This is indeed the case and the main result of~\cite{DePhilippisRindler16}:

\begin{theorem}\label{thm:main}
Let \(\Omega\subset \R^d\) be an open set, let \(\Acal\) be a $k$'th-order linear constant-coefficient differential operator as above, and let \(\mu \in \Mcal(\Omega;\R^N)\) be an \(\Acal\)-free Radon measure on $\Omega$ with values in $\R^N$. Then,
\[
\frac{\di\mu}{\di|\mu|}(x)\in \Lambda_{\Acal}  \qquad\textrm{for \(|\mu|^s\)-a.e.\ \(x\in \Omega\).}
\]
\end{theorem}

\begin{remark}\label{rem:rhs}
Theorem~\ref{thm:main} is also valid in  the situation
\begin{equation}\label{gdpeq:rhs}
\Acal\mu =\sigma \qquad\text{for some $\sigma\in \Mcal(\Omega;\R^n)$.}
\end{equation}
This can be reduced to the setting of Theorem~\ref{thm:main} by defining  \(\tilde \mu := (\mu,\sigma)\in \Mcal(\R^d;\R^{N+n})\) and  \(\tilde {\Acal}\) (with an additional $0$'th-order term) such that \eqref{gdpeq:rhs} is equivalent to \(\tilde {\Acal} \tilde \mu=0\). It is easy to check that \(\Lambda_{\tilde{\Acal}} = \Lambda_{\Acal} \times \R^n\) and that for \(|\mu|\)-almost every point \(\frac{\di \mu}{\di |\mu|}\) is proportional to \(\frac{\di \mu}{\di |\tilde{\mu}|}\).
\end{remark}

One interesting feature of Theorem~\ref{thm:main} is that it gives information about the directional structure of $\mu$ at singular points (the ``shape of singularities''). Indeed, it is not hard to check that for all ``elementary'' \(\Acal\)-free measures of the form
\begin{equation}\label{GDPFRelem}
\mu=\lambda \nu, \qquad \text{where}\qquad
\lambda\in \Lambda_\Acal,\; \nu\in \Mcal^+(\R^d),
\end{equation}
the scalar measure \(\nu\) is necessarily translation invariant along directions that are orthogonal to  the \emph{characteristic set}
\[
\Xi(\lambda) := \setb{\xi \in \R^d}{\lambda \in \ker \Abb(\xi)}.
\]
Note that \(\Xi(\lambda)\)  turns out to be a subspace of $\R^d$ whenever $\Acal$ is a first-order operator.  In this case, the translation invariance of \(\nu\) in the directions orthogonal to \(\Xi(\lambda)\) is actually the best information one can get from \eqref{GDPFRelem}.

In the case of operators of order  $k > 1$,  due to the lack of linearity of the map $\xi \mapsto \Abb^k(\xi)$ for $k > 1$, the structure of elementary \(\Acal\)-free measures is more complicated and not yet fully understood.

In the next sections we will describe some applications of  Theorem~\ref{thm:main} to the following problems:
\begin{itemize}
\item The description of the singular part of derivatives of \(\BV\)- and \(\BD\)-maps.
\item Lower semicontinuity for integral functionals defined on measures.
\item Characterization of generalized gradient Young measures.
\item The study of the sharpness of the Rademacher's theorem .
\item Cheeger's conjecture on Lipschitz differentiability spaces.
\end{itemize}
In Section~\ref{sc:proof} we will  sketch  the proof of Theorem~\ref{thm:main}.

\section{Structure of singular derivatives} \label{gdp:sec1}

 Let  \(f \colon \mathbb R^\ell\otimes \mathbb R^d\to [0,\infty)\) be a linear growth integrand with \(f(A)\sim |A|\) for \(|A|\) large. Consider the following variational problem:
\[
\int_{\Omega} f(\nabla u) \dd x \to \min,  \qquad u\in \Crm^1(\Omega;\R^\ell) \textrm{ with given boundary conditions}.
\]
It is well known that in order to apply the Direct Method of the calculus of variations one has to relax the above problem to a setting where it is possible to obtain both compactness  of minimizing sequences and lower semicontinuity of the functional with respect to some topology, usually the weak(*) topology in some function space. Due to the linear growth of the integrand the only easily  available estimate  on a minimizing sequence \((u_k)\) is an a-priori bound on the \(\Lrm^1\)-norm of their derivativesx: 
\[
\sup_{k}\int_\Omega |\nabla u_k| \dd x<\infty.
\]
It is then quite natural to relax the functional to the space \(\BV(\Omega, \R^\ell)\) of functions of \emph{bounded variation}, i.e.\ those functions \(u\in \Lrm^1(\Omega;\R^\ell)\) whose distributional gradient is a matrix-valued Radon measure. A fine understanding of the possible behavior of measures arising as derivatives is then fundamental to  study the weak* lower semicontinuity of the functional as well as its relaxation to the space \(\BV\).

In this respect, in \cite{AmbrosioDeGiorgi88} Ambrosio and De~Giorgi proposed the following conjecture:

\begin{question}
Is the singular part of the derivative of a  function  \(u\in \BV(\Omega;\R^\ell)\), which is usually denoted by \(D^su\),  always of rank one? Namely, is it true that 
\[
\frac{\di D^s u}{\di |D^s u|}(x)=a(x)\otimes b(x)
\]
for \(|D^s u|\)-a.e.\ \(x\) and some $a(x) \in \R^\ell$, $b(x) \in \R^d$?
\end{question}

Their conjecture was motivated by the fact that this structure is trivially true  for the so-called \emph{jump part} of $D^su$ (which is always of the form $[u] \otimes n \, \Hcal^{d-1} \restrict J$, where  $J$ is the $\Hcal^{d-1}$-rectifiable jump set, $n$ is a normal on $J$, and $[u]$ is the jump height in direction $n$); see \cite[Chapter 3]{AmbrosioFuscoPallara00book} for a complete reference concerning functions of bounded variations.

A positive answer to the above question was  given by Alberti in~\cite{Alberti93} with his celebrated \emph{rank-one theorem}. It was recognized quickly that this result has a central place in the calculus of variations and importance well beyond, in particular because it implies that locally all singularities in BV-functions are necessarily \emph{one-directional}. Indeed, after a blow-up (i.e.\ magnification) procedure at $\abs{D^s u}$-almost every point, the blow-up limit measure depends only on a single direction and is translation-invariant with respect to all orthogonal directions. This is not surprising for jumps, but it is a strong assertion about all other singularities in the \emph{Cantor part} of $D^s u$, i.e., the remainder of $D^s u$ after subtracting the jump part.

While Alberti's original proof is geometric in nature, one can also interpret the theorem as a result about singularities in PDEs: BV-derivatives $Du$ satisfy the  PDE
\[
  \curl Du = 0  \qquad\text{in the sense of distributions,}
\]
which can be written with a linear constant-coefficient PDE operator $\Acal := \sum_{j=1}^d A_j \partial_j$ as $\Acal \mu = 0$ in the sense of distributions.

Besides its intrinsic theoretical interest, the rank-one theorem also has many applications in the theory of \(\BV\)-functions, for instance for lower semicontinuity and relaxation~\cite{AmbrosioDalMaso92,FonsecaMuller93,KristensenRindler10b}, integral representation theorems~\cite{BouchitteFonsecaMascarenhas98}, Young measure theory~\cite{KristensenRindler10, Rindler14},  and the study of continuity equations with BV-vector fields~\cite{Ambrosio04}. We refer to~\cite[Chapter~5]{AmbrosioFuscoPallara00book} for further history.

At the the end of this section we will see that Alberti's rank-one theorem is a straightforward consequence of Theorem~\ref{thm:main}. Let us also mention that recently a very short   proof of the  Alberti rank-one theorem has been given by Massaccesi and Vittone in~\cite{MassaccesiVittone16}.

In problems arising in the theory of geometrically-linear elasto-plasticity~\cite{Suquet78,Suquet79,TemamStrang80} one often needs to consider a larger space of functions than the space of functions of bounded variations. Indeed, in this setting energies usually only depend on the \emph{symmetric part} of the gradient and one has to consider the following type of variational problem:
\begin{equation*}
\int_{\Omega} f(\Ecal u)\dd x \to \min,  \qquad u\in \Crm^1(\Omega;\R^d)\textrm{ with given boundary conditions},
\end{equation*} 
where \(\Ecal u=(\nabla u+ \nabla u^{T})/2\in \Rdds\) is the symmetric gradient ($\Rdds$ being canonically isomorphic to the space of symmetric $(d \times d)$-matrices) and  \(f \) is  a linear-growth integrand with \(f (A)\sim |A|\) for $|A|$ large. In this case, for a minimizing sequence \((u_k)\) one can only obtain that 
\[
\sup_{k}\int_\Omega |\Ecal u_k|\dd x<\infty
\]
and, due to the failure of Korn's inequality in \(\Lrm^1\)~\cite{Ornstein62,ContiFaracoMaggi05,KirchheimKristensen16}, this is not enough to ensure that \(\sup_k \int | \nabla u_k| \,\di x<\infty\). One then introduces the space \(\BD(\Omega)\) of functions of \emph{bounded deformation}, i.e.\ those functions \(u\in \Lrm^1(\Omega;\R^d)\) such that the symmetrized distributional derivative exists as a Radon measure, i.e.,
\[
  Eu := \frac{1}{2}(Du+Du^T) \in \Mcal(\Omega;(\R^d\otimes \R^d)_\sym),
\]
see~\cite{AmbrosioCosciaDalMaso97,Temam85book,TemamStrang80}. Clearly, \(\BV(\Omega,\R^d)\subset \BD(\Omega)\) and  the inclusion is strict~\cite{ContiFaracoMaggi05,Ornstein62}. Note that for $u \in \BV(\Omega;\R^d)$ as a consequence of Alberti's rank-one theorem one has
\[
\frac{\di E^s u}{\di |E^s u|}(x)=a(x)\odot  b(x),
\]
where \(a\odot b=(a\otimes b+b\otimes a)/2\) is the symmetrized tensor product. One is then naturally led to the following conjecture:

\begin{question}\label{q:BD}
Is it true that for every function \(u\in \BD(\Omega)\) it holds that
\[
\frac{\di E^s u}{\di |E^s u|}(x)=a(x)\odot  b(x)
\]
for \(|E^s u|\)-a.e.\ \(x\) and some $a(x), b(x) \in \R^d$?
\end{question}

Again, besides its theoretical interest, it has been well known that a positive answer of the above question would have several applications to the study of lower semicontinuity  and relaxation of functionals defined on \(\BD\), see the next section, as well as in establishing the absence of a singular part for minimizers, see for instance~\cite[Remark 4.8]{FrancfortGiacominiMarigo15}. 

Let us conclude this section by   showing  how both a positive answer to Question~\ref{q:BD} and a new proof of Alberti's rank-one theorem can easily be obtained by applying  Theorem~\ref{thm:main} to suitable differential operators:

\begin{theorem} \label{cor:main}
Let \(\Omega\subset \R^d\) be open. Then:
\begin{enumerate}[(i)]
\item If \(u\in \BV(\Omega;\R^\ell)\), then
\[
\frac{\di D^s u}{\di |D^s u|}(x)=a(x)\otimes b(x) \qquad\textrm{for \(|D^s u|\)-a.e.\ \(x\) and some $a(x) \in \R^\ell$, $b(x) \in \R^d$.}
\]
\item  If \(u\in \BD(\Omega)\), then
\[
\frac{\di E^s u}{\di |E^s u|}(x)=a(x)\odot b(x) \qquad\textrm{for \(|E^s u|\)-a.e.\ \(x\) and some $a(x), b(x) \in \R^d$.}
\]
\end{enumerate}
\end{theorem}

\begin{proof}
Observe that \(\mu=Du\) is curl-free,
\[
0=\curl \mu=\Big(\partial_i \mu^{k}_{j}-\partial_j  \mu^{k}_{i}\Big)_{i,j=1,\ldots, d;\; k=1,\ldots,\ell} \; .
\]
Then, assertion~(i) above follows from 
\[
\Lambda_{\curl}=\setb{ a\otimes \xi }{ a\in \R^\ell,\, \xi \in \R^d \setminus \{0\}},
\]
which can be proved by an easy computation.

In the same way, if \(\mu=Eu\), then $\mu$ satisfies the \emph{Saint-Venant compatibility conditions},
\[
0=\curl\curl \mu:= \biggl( \sum_{i=1}^d \partial_{ik} \mu_{i}^j+\partial_{ij} \mu_{i}^k-\partial_{jk} \mu_{i}^i-\partial_{ii} \mu_{j}^k \biggr)_{j,k=1,\ldots,d} \; .
\]
It is now a direct computation to check that 
\[
\Lambda_{\curl\curl}=\setb{ a\odot \xi}{  a\in \R^d,\, \xi \in \R^d \setminus \{0\} }.
\]
This shows assertion~(ii) above.
\end{proof}

%

\section{Functionals on measures}

The theory of integral functionals with linear-growth integrands defined on vector-valued measures satisfying PDE constraints is central to many questions of the calculus of variations. In particular, their relaxation and lower semicontinuity properties have attracted a lot of attention, see for instance~\cite{AmbrosioDalMaso92,FonsecaMuller93,FonsecaMuller99,FonsecaLeoniMuller04,KristensenRindler10b,Rindler11,BaiaChermisiMatiasSantos13}. Based on Theorem~\ref{thm:main} one can unify and extend many of these results.

Concretely, let $\Omega \subset \R^d$ be an open and bounded set and consider the functional
\begin{equation}\label{eq:F}
 \Fcal[\mu] :=  \int_\Omega f \biggl(x,\frac{\di\mu}{\di\Lcal^d}(x)\biggr) \dd x + \int_\Omega f^\infty \biggl(x,\frac{\di\mu^s}{\di|\mu|^s}(x)\biggr) \dd |\mu|^s(x),
\end{equation}
defined for finite vector Radon measures $\mu \in \Mcal(\Omega;\R^N)$ with values in $\R^N$ and satisfying
\[
  \Acal \mu  = 0  \qquad
  \text{in the sense of distributions.}
\]
Here, $f \colon {\Omega} \times \R^N \to [0,\infty)$ is a Borel integrand that has \emph{linear growth at infinity}, i.e.,
\[
  \abs{f(x,A)} \leq M(1+\abs{A})  \qquad
  \text{for all $(x,A) \in {\Omega} \times \R^N$.}
\]
We also assume that the \emph{strong recession function} of $f$ exists, which is defined as
\begin{equation}\label{eq:f_infty}
f^\infty(x,A) := \lim_{\substack {x' \to x \\ A' \to A \\ t \to \infty}} \frac{f(x',tA')}{t}, \qquad (x,A) \in \cl{\Omega} \times \R^N.
\end{equation}

The (weak*) lower semicontinuity properties of $\Fcal$ depend on (generalized) \emph{convexity} properties of the integrand in its second variable. For this, we need the following definition: A Borel function $h \colon \R^N \to \R$ is called \emph{$\Acal^k$-quasiconvex} ($\Acal^k=\sum_{|\alpha|=k}A_\alpha\partial^\alpha $ being the principal part of $\Acal$) if
\[
  h(F) \leq \int_Q h(F + w(y)) \dd y
\]
for all $F \in \R^N$ and all $Q$-periodic $w \in \Crm^\infty(Q;\R^N)$ such that $\Acal^k w = 0$ and $\int_Q w \dd y = 0$, where $Q := (0,1)^d$ is the open unit cube in $\R^d$; see~\cite{FonsecaMuller99} for more on this class of integrands. For $\Acal = \curl$ this notion is equivalent to the classical \emph{quasiconvexity} as introduced by Morrey~\cite{Morrey52}.

It has been known for a long time that $\Acal^k$-quasiconvexity of $f(x,\frarg)$ is a necessary condition for the sequential weak* lower semicontinuity of $\Fcal$ on $\Acal$-free measures. As for the sufficiency, we can now prove the following general lower semicontinuity theorem, which is taken from~\cite{ArroyoRabasaDePhilippisRindler17?} (where also more general results can be found):

\begin{theorem} \label{thm:lsc}
Let $f \colon \Omega \times \R^N \to [0,\infty)$ be a continuous integrand with linear growth  at infinity such that $f$ is uniformly Lipschitz in its second argument, $f^\infty$ exists as in~\eqref{eq:f_infty}, and $f(x,\frarg)$ is $\Acal^k$-quasiconvex for all $x \in \Omega$. Further assume that there exists a modulus of continuity $\omega \colon [0,\infty) \to [0,\infty)$ (increasing, continuous, $\omega(0) = 0$) such that
\begin{equation}\label{eq:modulus}
  |f(x,A) - f(y,A)|\le \omega(|x - y|)(1 + |A|)  \qquad
\text{for all $x,y \in {\Omega}$, $A \in \R^N$.}
\end{equation}
Then, the functional $\Fcal$ is sequentially weakly* lower semicontinuous on the space
\[
  \Mcal(\Omega;\R^N) \cap \ker \Acal := \setb{ \mu \in \Mcal(\Omega;\R^N) }{ \Acal \mu = 0 }.
\]
\end{theorem}

\begin{remark} \label{rem:special}
As special cases of Theorem~\ref{thm:lsc} we get, among others, the following well-known results:

\begin{enumerate}[(i)]
\item For $\Acal = \curl$, one obtains  BV-lower semicontinuity results in the spirit of  Ambrosio--Dal Maso~\cite{AmbrosioDalMaso92} and Fonseca--M\"{u}ller~\cite{FonsecaMuller93}.

\item For  $\Acal = \curl \curl$, the second order operator expressing the Saint-Venant compatibility conditions, we re-prove the lower semicontinuity and relaxation theorem in the space of functions of bounded deformation (BD) from~\cite{Rindler11}.
\item For first-order operators $\Acal$, a  similar result was proved in~\cite{BaiaChermisiMatiasSantos13}.
\end{enumerate}
\end{remark}

The proof of Theorem~\ref{thm:lsc} essentially follows by combining Theorem~\ref{thm:main} with the main theorem of~\cite{KirchheimKristensen16}, which  establishes that the restriction of $f^\infty$ to the linear space spanned by the wave cone is in fact \emph{convex} at all points of  \(\Lambda_\Acal\) (in the sense that a supporting hyperplane exists). In this way we gain classical convexity for  $f^\infty$ at singular points, which can be exploited via the theory of 
generalized Young measures developed in~\cite{DiPernaMajda87,AlibertBouchitte97,KristensenRindler10} and also briefly discussed in the next section.

One can also show relaxation results, where $f$ is not assumed to be $\Acal^k$-quasiconvex in the second argument and the task becomes to compute the largest weakly* lower semicontinuous functional below $\Fcal$; see~\cite{ArroyoRabasaDePhilippisRindler17?} for more details.

\section{Characterization of generalized Young measures}

Young measures quantitatively describe the asymptotic oscillations in $\Lrm^p$-weakly converging sequences. They were introduced in~\cite{Young37,Young42a,Young42b} and later developed into an important tool in modern PDE theory and the calculus of variations in~\cite{Tartar79,Tartar83,Ball89,BallJames87} and many other works. In order to deal with concentration effects as well, DiPerna \& Majda extended the framework to so-called \enquote{generalized} Young measures, see~\cite{DiPernaMajda87,AlibertBouchitte97,KruzikRoubicek97,FonsecaMullerPedregal98,Sychev99,KristensenRindler10}. In the following we will refer also to these objects simply as \enquote{Young measures}. We recall some basic theory, for which proofs and examples can be found in~\cite{AlibertBouchitte97,KristensenRindler10,Rindler11}.

Let again $\Omega \subset \R^d$ be a bounded Lipschitz domain. For $f \in \Crm(\cl{\Omega} \times \R^N)$ we define
\[
  \Ebf(\Omega;\R^N) := \setb{ f \in \Crm(\cl{\Omega} \times \R^N) }{ \text{$f^\infty$ exists in the sense~\eqref{eq:f_infty}} }.
\]

A \emph{(generalized) Young measure} $\nu \in \Ybf(\Omega;\R^N) \subset \Ebf(\Omega;\R^N)^*$ on the open set $\Omega \subset \R^d$ with values in $\R^N$ is a triple $\nu = (\nu_x,\lambda_\nu,\nu_x^\infty)$ consisting of
\begin{enumerate}[(i)]
  \item a parametrized family of probability measures $(\nu_x)_{x \in \Omega} \subset \Mcal_1(\R^N)$, called the \emph{oscillation measure};
  \item a positive finite measure $\lambda_\nu \in \Mcal_+(\cl{\Omega})$, called the \emph{concentration measure}; and
  \item a parametrized family of probability measures $(\nu_x^\infty)_{x \in \cl{\Omega}} \subset \Mcal_1(\Sbb^{N-1})$, called the \emph{concentration-direction measure},
\end{enumerate}
for which we require that
\begin{enumerate}[(i)]
  \item[(iv)] the map $x \mapsto \nu_x$ is \emph{weakly* measurable} with respect to $\Lcal^d$, i.e.\ the function $x \mapsto \dpr{f(x,\frarg),\nu_x}$ is $\Lcal^d$-measurable for all bounded Borel functions $f \colon \Omega \times \R^N \to \R$,
  \item[(v)] the map $x \mapsto \nu_x^\infty$ is weakly* measurable with respect to $\lambda_\nu$, and
  \item[(vi)] $x \mapsto \dprn{\abs{\frarg},\nu_x} \in \Lrm^1(\Omega)$.
\end{enumerate}

The \emph{duality pairing} between $f \in \Ebf(\Omega;\R^N)$ and $\nu \in \Ybf(\Omega;\R^N)$ is given as
\begin{align*}
  \ddprb{f,\nu} &:= \int_\Omega \dprb{f(x,\frarg), \nu_x} \dd x
    + \int_{\cl{\Omega}} \dprb{f^\infty(x,\frarg),\nu_x^\infty} \dd \lambda_\nu(x) \\
  &:= \int_\Omega \int_{\R^N} f(x,A) \dd \nu_x(A) \dd x
  + \int_{\cl{\Omega}} \int_{\partial \Bbb^N} f^\infty(x,A) \dd \nu_x^\infty(A) \dd \lambda_\nu(x).
\end{align*}
If $(\gamma_j) \subset \Mcal(\cl{\Omega};\R^N)$ is a sequence of Radon measures with $\sup_j \abs{\gamma_j}(\cl{\Omega}) < \infty$, then we say that the sequence $(\gamma_j)$ \emph{generates} a Young measure $\nu \in \Ybf(\Omega;\R^N)$, in symbols $\gamma_j \toY \nu$, if for all $f \in \Ebf(\Omega;\R^N)$ it holds that
\begin{align*}
&f \biggl( x, \frac{\di \gamma_j}{\di \Lcal^d}(x)\biggr) \,\Lcal^d \restrict \Omega
+ f^\infty \biggl(x, \frac{\di \gamma^s_j}{\di \abs{\gamma^s_j}}(x) \biggr) \, \abs{\gamma^s_j}(\di x) \\
&\qquad\toweakstar\;\; \dprb{f(x,\frarg), \nu_x} \, \Lcal^d \restrict \Omega + \dprb{f^\infty(x,\frarg),
\nu_x^\infty} \, \lambda_\nu(\di x)  \qquad\text{in $\Mcal(\cl{\Omega})$.}
\end{align*}
Here, $\gamma_j^s$ is the singular part of $\gamma_j$ with respect to Lebesgue measure.





It can be shown that if $(\gamma_j) \subset \Mcal(\cl{\Omega};\R^N)$ is a sequence of measures with $\sup_j \abs{\gamma_j}(\cl{\Omega}) < \infty$ as above, then there exists a subsequence (not relabeled) and a Young measure $\nu \in \Ybf(\Omega;\R^N)$ such that $\gamma_j \toY \nu$, see~\cite{KristensenRindler10}.

When considering generating sequences $(\gamma_j)$ as above that satisfy a differential constraint like curl-freeness (i.e.\ the generating sequence is a sequence of \emph{gradients}), the following question arises:

\begin{question}
Can one characterize the class of Young measures generated by sequences satisfying some (linear) PDE constraint?
\end{question}

In applications, such results provide valuable information on the allowed oscillations and concentrations that are possible under this differential constraint, which usually constitutes a strong restriction. Characterization theorems are of particular use in the relaxation of minimization problems for non-convex integral functionals, where one passes from a functional defined on functions to one defined on Young measures. A characterization theorem then allows one to restrict the class of Young measures over which to minimize. This strategy is explained in detail (for classical Young measures) in~\cite{Pedregal97book}. 

The first general classification results are due to Kinderlehrer \& Pedregal~\cite{KinderlehrerPedregal91,KinderlehrerPedregal94}, who characterized classical \emph{gradient} Young measures, i.e.\ those generated by gradients of $\Wrm^{1,p}$-bounded sequences, $1 < p \leq \infty$. Their theorems put such gradient Young measures in duality with quasiconvex functions. For generalized Young measures the corresponding result was proved in~\cite{FonsecaMullerPedregal98} (also see~\cite{KalamajskaKruzik08}) and numerous other characterization results in the spirit of the Kinderlehrer--Pedregal theorems have since appeared, see for instance~\cite{KruzikRoubicek96,FonsecaMuller99,FonsecaKruzik10,BenesovaKruzik16}.

The characterization of generalized BV-Young measures, i.e.\ those $\nu$ generated by a sequence $(Du_j)$ of the $\BV$-derivatives of maps $u_j \in \BV(\Omega;\R^\ell)$) was first achieved in~\cite{KristensenRindler10}. A different, \enquote{local} proof was given in~\cite{Rindler14}, another improvement is in~\cite[Theorem~6.2]{KirchheimKristensen16}. All of these arguments crucially use Alberti's rank-one theorem.

The most interesting case beyond BV is again the case of functions of bounded deformation (BD), which were introduced above: In plasticity theory~\cite{Suquet78,Suquet79,TemamStrang80}, one often deals with sequences of uniformly $\Lrm^1$-bounded symmetric gradients $\Ecal u_j := (\nabla u_j + \nabla u_j^T)/2$. In order to understand the asymptotic oscillations and concentrations in such sequences  $(\Ecal u_j)$ one needs to characterize the (generalized) Young measures $\nu$ generated by them. We call such $\nu$ \emph{BD-Young measures} and write $\nu \in \BDY(\Omega)$, since all BD-functions can be reached as weak* limits of sequences $(u_j)$ as above.



In this situation the following result can be shown, see~\cite{DePhilippisRindler17}:

\begin{theorem} \label{thm:BDY_charact}
Let $\nu \in \Ybf(\Omega;\Rdds)$ be a (generalized) Young measure. Then, $\nu$ is a BD-Young measure, $\nu \in \BDY(\Omega)$, if and only if there exists $u \in \BD(\Omega)$ with
\[
  \dprb{\id,\nu_x} \, \Lcal^d_x + \dprb{\id,\nu_x^\infty} \,(\lambda_\nu \restrict \Omega)(\di x) = Eu
\]
and for all symmetric-quasiconvex $h \in \Crm(\Rdds)$ with linear growth at infinity, the \emph{Jensen-type inequality}
\[
  h \biggl( \dprb{\id,\nu_x} + \dprb{\id,\nu_x^\infty} \frac{\di \lambda_\nu}{\di \Lcal^d}(x) \biggr)
    \leq \dprb{h,\nu_x} + \dprb{h^\#,\nu_x^\infty} \frac{\di \lambda_\nu}{\di \Lcal^d}(x).
\]
holds at $\Lcal^d$-almost every $x \in \Omega$, where $h^\#$ is defined via
\[
h^\#(A) := \limsup_{\substack {A' \to A \\ t \to \infty}} \frac{h(tA')}{t}, \qquad A \in \Rdds.
\]
\end{theorem}

One application of this result (in the spirit of Young's original work~\cite{Young42a,Young42b,Young80book}) is the following: For a suitable integrand $f \colon \Omega \times \Rdds \to \R$, the minimum principle
\begin{equation} \label{eq:minprinc_ext}
  \ddprb{f,\nu} \to \min,  \quad\text{$\nu \in \BDY(\Omega)$.}
\end{equation}
can be seen as the \emph{extension-relaxation} of the minimum principle 
\begin{equation} \label{eq:minprinc_orig}
  \int_\Omega f(x,\Ecal u(x)) \dd x + \int_\Omega f^\infty \biggl(x,\frac{\di E^s u}{\di \abs{E^s u}}(x)\biggr) \dd \abs{E^s u} \to \min
\end{equation}
over $u \in \BD(\Omega)$. The point is that~\eqref{eq:minprinc_orig} may not be solvable if $f$ is not symmetric-quasiconvex, whereas~\eqref{eq:minprinc_ext} always has a solution. In this situation, Theorem~\ref{thm:BDY_charact} then gives (abstract) restrictions on the Young measures to be considered in~\eqref{eq:minprinc_ext}. Another type of relaxation involving the symmetric-quasiconvex envelope of $f$ is investigated in~\cite{ArroyoRabasaDePhilippisRindler17?} within the framework of general linear PDE side-constraints.

The necessity part of Theorem~\ref{thm:BDY_charact} follows from a lower semicontinuity or relaxation theorem like the one in~\cite{Rindler11}. For the sufficiency part (which is quite involved), one first characterizes so-called \emph{tangent Young measures}, which are localized versions of Young measures. There are two types: regular and singular tangent Young measures, depending on whether regular (Lebesgue measure-like) effects or singular effects dominate around the blow-up point. We stress that the argument crucially rests on the BD-analogue of Alberti's rank-one theorem, see Corollary~\ref{cor:main}~(ii). Technically, in one of the proof steps to establish Theorem~\ref{thm:BDY_charact} we need to create \enquote{artificial concentrations} by compressing symmetric gradients in one direction. This is only possible if we know precisely what these singularities look like.

A characterization results for Young measures under general linear PDE constrainsts is currently not available (there is a partial result in the work~\cite{BaiaMatiasSantos13}, but limited to first-order operators and needing additional technical assumptions). The reason is that currently not enough is known about the directional structure of $\Acal$-free measures at singular points.

\section{The converse of Rademacher's theorem}\label{gdp:sec2}

Rademacher's theorem asserts that  a Lipschitz function  \(f\in \Wrm^{1,\infty}(\R^d,\R^\ell)\) is diffferentiable \(\Lcal^d\)-almost everywhere. A natural question, which has attracted considerable attention,  is to understand how sharp this result is. The following questions have been folklore in the area for a while:

\begin{question}[Strong converse of Rademacher's theorem]\label{q:Rademacher_strong}
Given a Lebesgue null set \(E\subset \R^d\) is it possible to find some \(\ell\ge 1\) and a Lipschitz function \(f\in \Wrm^{1,\infty}(\R^d,\R^\ell)\) such that \(f\) is not differentiable in   any point of \(E\)?
\end{question}

\begin{question}[Weak converse of Rademacher's theorem]\label{q:Rademacher_weak} Let \(\nu\in \Mcal_+(\R^d)\) be a positive Radon measure such that every Lipschitz function is differentiable \(\nu\)-almost everywhere. Is  \(\nu\) necessarily absolutely continuous with respect to $\Lcal^d$?
\end{question}
Clearly, a positive answer to Question~\ref{q:Rademacher_strong} implies a positive answer to Question~\ref{q:Rademacher_weak}. Let us also stress that   in answering  Question~\ref{q:Rademacher_strong}, an important role is played by the dimension \(\ell\) of the target set, see point~(ii) below,  while this does not have any influence on Question~\ref{q:Rademacher_weak}, see~\cite{AlbertiMarchese16}. We refer to~\cite{AlbertiCsornyeiPreiss05, AlbertiCsornyeiPreiss10, AlbertiMarchese16} for a detailed account on the history of these problems and here we simply record the following facts:
\begin{enumerate}[(i)]
\item For \(d=1\) a positive answer to Question~\ref{q:Rademacher_strong} is due to Zahorski~\cite{Zahorski46}.
\item For \(d\ge 2\) there exists a null set \(E\) such that every  Lipschitz function \(f \colon \R^{d}\to \R^\ell\) with $\ell < d$ is differentiable in at least one point of \(E\). This is was proved by  Preiss in~\cite{Preiss90}   for \(d=2\) and later extended by  Preiss and Speight in~\cite{PreissSpeight14} to every dimension.
\item For \(d=2\) a positive answer to Question~\ref{q:Rademacher_strong} has been  given by Alberti, Cs\"ornyei and Preiss as  a consequence of their  deep  result concerning the structure of null sets in the plane~\cite{AlbertiCsornyeiPreiss05, AlbertiCsornyeiPreiss10,AlbertiCsornyeiPreissToAppear}. Namely, they show that for every null set \(E\subset \R^2\) there exists a Lipschitz function \(f \colon \R^2\to \R^2\) such that \(f\) is not differentiable at any point of \(E\).
\item For \(d\ge2\)  an extension of the result described in point~(iii) above, i.e.\ that for every null set \(E\subset \R^d\) there exists a Lipschitz function \(f \colon \R^d\to \R^d\) such that \(f\) is not differentiable at any point of \(E\), has been announced in 2011 by Cs\"ornyei and Jones~\cite{Jones11talk}.
\end{enumerate}

Let us now  show  how Question~\ref{q:Rademacher_weak} is  related to Question~\ref{q:Afree}. In~\cite[Theorem~1.1]{AlbertiMarchese16} Alberti \& Marchese have shown the following result:

\begin{theorem}[Alberti--Marchese] Let \(\nu\in \Mcal_+(\R^d)\) be a positive Radon measure. Then, there exists a vector space-valued \(\nu\)-measurable map \(V(\nu,x)\) (the \emph{decomposability bundle of \(\nu\)}) such that:
\begin{enumerate}[(i)]
\item Every Lipschitz function \(f \colon \R^d\to \R\) is differentiable in the directions of \(V(\nu,x)\) at \(\nu\)-almost every \(x\).
\item There exists a Lipschitz function \(f \colon \R^d\to \R\) such that for \(\nu\)-almost every \(x\) and every \(v\notin V(\nu,x)\) the derivative of \(f\) at \(x\) in the direction of \(v\) does not exist.
\end{enumerate}
\end{theorem}

Thanks to the above theorem, Question~\ref{q:Rademacher_weak} is then  equivalent to the following:

\begin{question}\label{gdp:q3}Let \(\nu\in \Mcal_+(\R^d)\) be a positive Radon measure such that \(V(\nu,x)=\R^d\) for \(\nu\)-almost every \(x\). Is \(\nu\) absolutely continuous with respect to $\Lcal^d$?
\end{question}

The link between the above question and Theorem~\ref{thm:main} is due to the following result, again due to Alberti \& Marchese, see~\cite[Corollary 6.5]{AlbertiMarchese16} and~\cite[Lemma 3.1]{DePhilippisRindler16}\footnote{In the cited  references the results are stated in terms of normal currents. By the trivial identifications of the space of normal currents with the space of measure-valued vector fields whose divergence is a  measure  it is immediate to see that they are  equivalent to our Lemma~\ref{lem:bundle}}.

\begin{lemma}\label{lem:bundle}
Let \(\nu\in\Mcal_+(\R^d)\) be  a positive Radon measure. Then the  following are equivalent:
\begin{enumerate}
\item The decomposability bundle of \(\nu\) is of full dimension, i.e.\ \(V(\nu,x)=\R^d\) for \(\nu\)-almost every \(x\).
\item There exist \(\R^d\)-valued measures  \(\mu_1,\dots,\mu_d\in \Mcal (\R^d;\R^d)\)  with measure-valued divergences \( \diverg \mu_i\in \Mcal(\R^d;\R)\)  such that  \(\nu\ll |\mu_i |\) for \(1=1,\dots,d\) and\footnote{Note that since \(\nu\ll |\mu_i |\) for all \(i=1,\dots,d\), in item (ii) above all the Radon-Nikodym derivatives \(\frac{\di \mu_i}{\di|\mu_i|}\) \(i=1,\dots,d\) exist for \(\nu\)-a.e.\ \(x\). }  
 \begin{equation}\label{gdp:span}
 \spn\bigg \{\frac {\di \mu_1}{\di |\mu_1|}(x),\dots,\frac {\di \mu_d}{\di |\mu_d|}(x)\bigg\}=\R^d\qquad\textrm{for \(\nu\)-a.e.\ \(x\).}
\end{equation}
\end{enumerate}

\end{lemma}

With the above lemma at hand,  a positive answer to Question~\ref{gdp:q3} (and thus to Question~\ref{q:Rademacher_weak}) follows from Theorem~\ref{thm:main} in a straightforward fashion:

\begin{theorem} \label{thm:convRademacher}
Let \(\nu\in \Mcal_+(\R^d)\) be a positive Radon measure such that every Lipschitz function is differentiable \(\nu\)-almost everywhere. Then, \(\nu\) is absolutely continuous with respect to $\Lcal^d$.
\end{theorem}

\begin{proof}
Let \(\nu\) be a measure such that   \(V(\nu,x)=\R^d\) for \(\nu\)-almost every \(x\) an let \(\mu_i\) be the measures provided by  Lemma~\ref{lem:bundle}~(ii). Let us  consider the matrix-valued measure
\[
\boldsymbol{\mu}:=
\begin{pmatrix}
\mu_1  \\
\vdots\\
\mu_d \\
\end{pmatrix}
\in  \Mcal (\R^d;\R^d\otimes \R^d).
\]
Note that  \(\diverg \boldsymbol \mu\in \Mcal(\R^d;\R^d)\), where \(\diverg\) is the row-wise divergence operator. Since, by direct computation, 
\[
\Lambda_{\diverg}=\setb{ M\in \R^d\otimes \R^d }{ {\rm rank}\, M\le d-1 },
\] 
Theorem~\ref{thm:main} and Remark~\ref{rem:rhs} imply that \({\rm rank}\, \bigl(\frac{\di \boldsymbol{\mu}}{\di |\boldsymbol{\mu}|}\bigr)\le d-1\)  for \(|\boldsymbol{\mu}|^s\)-almost every point. Hence, by~\eqref{gdp:span}, \(\nu\) is singular with respect to \(|\boldsymbol{\mu}|^s\). On the other hand, since \(\nu\ll |\mu_i|\) for all \(i=1,\dots,d\), we get \(\nu^s\ll |\boldsymbol{\mu}|^s\). Hence, we conclude \(\nu^s=0\), as desired. 
\end{proof}

Let us conclude this section by remarking that the weak converse of Rademacher's theorem, i.e.\ a positive answer to Question~\ref{q:Rademacher_weak},  has some consequences for the structure of Ambrosio--Kirchheim metric currents~\cite{AmbrosioKirchheim00}, see the work of Schioppa~\cite{Schioppa16}. In particular, it allows one to prove the  top-dimensional case of the flat chain conjecture proposed by Ambrosio and Kirchheim in~\cite{AmbrosioKirchheim00}. 


\section{Cheeger's conjecture}
Among the many applications of Rademacher's theorem, it allows one to pass from ``non-infintesimal'' information (the existence of certain Lipschitz maps) to infinitesimal information. For instance, one can easily establish the following fact:

\medskip
\emph{There is no bi-Lipschitz map \(f:\R^d\to \R^\ell\) if \(d\ne \ell\)}.
\medskip

Indeed, if this were the case, Rademacher's theorem  would imply  (at a differentiability point) the existence of a bijective \emph{linear} map from \(\R^d\) to \(\R^\ell\) with  \(d\ne \ell\).  

While the above statement is an immediate consequence of the theorem on the invariance of dimension (asserting that there are no bijective continuous maps from \(\R^d\) to \(\R^\ell\) if \(d\ne \ell\)), the point here is that the almost everywhere result allows to pass from a non-linear statement (the existence of a bi-Lipschitz map) to a linear one, whose rigidity can be proved by elementary methods. 

This line of thought has been adopted in the study of rigidity of several metric structures.  For instance, the natural generalization of Rademacher's theorem  in the context of Carnot groups, which was established by Pansu in~\cite{Pansu89}, allows one to show that there are no bi-Lipscitz embeddings of a Carnot group into \(\R^\ell\) if the former is non-commutative.

The fact that (a suitable notion of)  differentiability of Lipschitz functions allows to develop a first-order calculus on metric spaces and in turn to obtain non-embedding results has been recognised by Cheeger in his seminal paper~\cite{Cheeger99}  and later studied by several authors.


Let us briefly introduce the theory of Cheeger as it has been axiomatized by Keith in~\cite{Keith04}. Note that  it is natural  to generalize Rademacher's theorem to the setting of \emph{metric measure spaces} since we need to talk about Lipschitz functions (a metric concept) and almost everywhere (a measure-theoretic concept).

Let  $(X,\rho,\mu)$ be a  \emph{metric measure space}, that is, $(X,\rho)$ is a separable, complete metric space and $\mu \in \Mcal_+(X)$ is a positive  Radon measure on $X$. We call a pair \((U,\phi)\) such that  $U \subset X$ is a Borel set and  $\phi \colon X \to \R^d$ is Lipschitz, a \emph{$d$-dimensional chart}, or simply a \emph{$d$-chart}. A function  $f \colon X \to \R$ is said to be \emph{differentiable with respect to a $d$-chart $(U,\phi)$}  at $x_0 \in U$ if there exists a unique (co-)vector $df(x_0) \in \R^d$ such that
\begin{equation} \label{eq:LDS_diff}
  \limsup_{x\to x_0} \frac{\abs{f(x) - f(x_0) - df(x_0) \cdot (\phi(x) - \phi(x_0))}}{\rho(x,x_0)} = 0.
\end{equation}

\begin{definition}A metric measure space $(X,\rho,\mu)$ is a \emph{Lipschitz differentiability space}  if there exists a countable family of \(d(i)\)-charts $(U_i,\phi_i)$ ($i \in \N$) such that $X=\bigcup_iU_i$ and any Lipschitz map $f \colon X \to \R$ is differentiable with respect to every $(U_i,\phi_i)$ at $\mu$-almost every point $x_0 \in U_i$.
\end{definition} 

In this terminology, the main result of~\cite{Cheeger99} asserts that  every doubling  metric measure space $(X,\rho,\mu)$  satisfying a Poincar\'{e} inequality is a Lipschitz differentiability space.

%
In the same paper, Cheeger conjectured that the push-forward of the reference measure \(\mu\) under every chart \(\varphi_i\) has to be absolutely continuous with respect to the Lebesgue measure, see~\cite[Conjecture 4.63]{Cheeger99}:

\begin{question}
For every $d$-chart $(U,\phi)$ in a Lipschitz differentiability space, does it hold that $\phi_\# (\mu \restrict U) \ll \Lcal^d$?
\end{question}

Some consequences of this fact concerning the existence of bi-Lipschitz embeddings of \(X\) into some \(\R^N\) are detailed in~\cite[Section 14]{Cheeger99}, also see~\cite{CheegerKleiner06,CheegerKleiner09}.

Let us assume that  \((X,\rho,\mu) = (\R^d,\rho_{\mathcal E}, \nu)\) with \( \rho_{\mathcal E}\) the Euclidean distance and \(\nu\) a positive Radon measure, is a Lipschitz differentiability space when equipped with the (single) identity chart (note that it follows a-posteriori from the validity of Cheeger's conjecture that no mapping into a higher-dimensional space can be a chart in a Lipschitz differentiability structure of $\R^d$). In this case the validity of Cheeger's conjecture reduces  to the validity of the  (weak) converse of Rademacher's theorem, which we stated above in Theorem~\ref{thm:convRademacher}.


One can also prove the assertion of Cheeger's conjecture directly. Indeed, from the work of Bate~\cite{Bate15}, and Alberti--Marchese~\cite{AlbertiMarchese16} an analogue of Lemma~\ref{lem:bundle} for $\varphi_\# (\mu\restrict U)$ in place of $\mu$ follows, see also~\cite{Schioppa16b,Schioppa16}. This allows one to conclude as in the proof of Theorem~\ref{thm:convRademacher} to get:

\begin{theorem} \label{thm:Cheeger}
Let $(X,\rho,\mu)$ be a Lipschitz differentiability space and let \((U,\varphi)\) be a \(d\)-dimensional chart. Then, 
\(
\varphi_\# (\mu\restrict U)\ll \Lcal^d.
\)
\end{theorem}

The details can be found in~\cite{DePhilippisMarcheseRindler17}.


We conclude this section by mentioning that the weak converse of Rademacher's theorem  also has some consequences concerning the structure of measures on metric measure spaces with Ricci curvature bounded from below, see~\cite{KellMondino16?, GigliPasqualetto16?}.

\section{Sketch of the proof of Theorem~\ref{thm:main}}\label{sc:proof}
In this section we shall give some details concerning the proof of Theorem~\ref{thm:main}. For simplicity we will only consider the case in which \(\Acal\) is a \emph{first-order homogeneous} operator, namely we will assume  that \(\mu\) satisfies 
\begin{equation*}
\Acal\mu =\sum_{j=1}^d A_j\partial_j \mu = 0  \qquad \text{in the sense of distributions.}
\end{equation*}
Note that in this case we have
\[
\Lambda_{\Acal} =\bigcup_{|\xi|=1}\ker \mathbb A(\xi),
\qquad \mathbb A(\xi)=\mathbb A^1(\xi) =2\pi {\rm i} \, \sum_{j=1}^d A_j \xi_j.
\]
Let 
\[
E:=\setBB{ x\in \Omega}{\frac{\di\mu}{\di|\mu|}(x)\notin \Lambda_{\Acal}},
\]
and let us assume by contradiction that \(|\mu|^s(E)>0\).

Employing a fundamental technique of geometric measure theory, one can ``zoom in'' around a generic point of \(E\). Indeed, one can show that for \(|\mu|^s\)-almost every  point \(x_0\in E\) there exists a sequence of radii \(r_k\downarrow 0\) such that 
\[   
\wslim_{k\to \infty} \frac{(T^{x_0,r_k})_\#\mu}{|\mu|(B_{r_k}(x_0))}=\wslim_{k\to \infty} \frac{(T^{x_0,r_j})_\#\mu^s}{|\mu|^s(B_{r_k}(x_0))} =P_0 \nu,
\]
where  \(T^{x,r} \colon \R^d\to \R^d\) is the dilation map \(T^{x,r}(y):=(y-x)/r\), \(T^{x,r}_\#\) denotes the push-forward operator\footnote{That is, for any measure \(\sigma\) and Borel set \(B\), \([(T^{x,r})_\#\sigma](B):=\sigma(x+rB)\)}, \(\nu\in {\rm Tan} (x_0,|\mu|)={\rm Tan} (x_0,|\mu|^s)\) is a non-zero tangent measure in the sense of Preiss~\cite{Preiss87}, 
\[
\lambda_0 =\frac{\di \mu}{\di |\mu|}(x_0)  \notin \Lambda_{\Acal},
\]
and the limit is to be understood in the weak* topology of Radon measures (i.e.\ in duality with compactly supported continuous functions). Moreover, one easily checks that
\[
\sum_{j=1}^d A_j  \frac{\di \mu}{\di |\mu|}(x_0)\,\partial_j \nu=0  
\qquad \text{in the sense of distributions.}
\]
By taking the Fourier transform  of the above equation, we get
\[
[\mathbb A (\xi) \lambda_0] \, \hat \nu(\xi)=0,  \qquad \xi \in \R^d.
\]
where  \(\hat \nu(\xi)\) is the Fourier transform of \(\nu\) in the sense of distributions (actually,  \(\nu\) does not need to  be a tempered distribution, hence some care is needed, see below for more details). Having assumed that \(\lambda_0\notin \Lambda_{\Acal}\), i.e.\ that 
\[
\mathbb A (\xi) \lambda_0\ne 0  \quad \text{for all \(\xi\ne 0\),} 
\]
 this implies \({\rm supp}\,\hat \nu=\{0\}\) and thus \(\nu\ll \Lcal^d\).  The latter fact, however, is not by itself a contradiction to 
\(\nu \in {\rm Tan} (x_0,|\mu|^s)\). Indeed, Preiss~\cite{Preiss87}  provided an example of a purely singular measure that has only multiples of Lebesgue measure as tangents (we also refer to~\cite{Oneill95} for a measure that has \emph{every} measure as a tangent at almost every point).

The above reasoning provides a sort of \emph{rigidity} property for \(\Acal\)-measures: If, for a  constant  polar vector  \(\lambda_0\notin \Lambda_{\Acal}\) and a measure  \(\nu\in \Mcal_+(\R^d)\), the measure  \(\lambda_0\nu\) is \(\Acal\)-free, then necessarily \(\nu\ll \Lcal^d\). However, as we commented above, this is not enough to conclude. In order to prove the theorem we need to strengthen this rigidity property (absolutely continuity of the measures \(\lambda_0\nu\) with \(\lambda_0\notin\Lambda_\Acal\))  to a stability  property which can be roughly stated as follows:

\medskip

{\em \(\Acal\)-free measures \(\mu\) with \(|\mu|\big(\big\{x: \frac{\di\mu}{\di|\mu|}(x) \in \Lambda_\Acal\big\}\big)\ll1\) have small singular part.}

\medskip

In this respect note that since  \(\lambda_0\notin \Lambda_{\Acal} \) implies that  \(\mathbb A (\xi)\lambda_0\ne 0\) for \(\xi\ne 0\), one can hope for some sort of ``elliptic regularization'' that forces not only \(\nu\ll \Lcal^d\) but also 
\[
\mu_k :=  \frac{(T^{x_0,r_k})_\#\mu}{|\mu|^s(B_{r_k}(x_0))}\ll  \Lcal^d,
\]
at least for small $r_k$. This is  actually the case: Inspired by Allard's strong constancy lemma in~\cite{Allard86}, we can show that  the ellipticity of the system at the limit (i.e.\ that \(\mathbb A (\xi)\lambda_0\ne 0\)) improves   the weak* convergence of \((\mu_k)\) to  convergence in the   total variation norm, i.e.\
\begin{equation}\label{gdp:totvar}
  |\mu_k-\lambda_0\nu|(B_{1/2})\to 0.
\end{equation}
 Since the singular part of \(\mu_k\) is asymptotically predominant around \(x_0\), see \eqref{gdp:yeah} below, this latter fact  implies that 
\begin{equation*}
  |\mu^s_k-\lambda_0\nu|(B_{1/2})\to 0,
\end{equation*}
which easily gives a contradiction to $\nu \ll \Lcal^d$ and concludes the proof.
 
Let us briefly sketch   how \eqref{gdp:totvar} is obtained. For \(\chi\in \Dcal(B_1)\), \(0\le \chi\le 1\), consider the measures $\lambda_0\chi \nu_k$, where
\[
 \nu_k :=\frac{(T^{x_0,r_k})_\#|\mu|^s}{|\mu|^s(B_{r_k}(x_0))},
\]
and note that, since we can assume that for the chosen \(x_0\) it holds that 
\begin{equation}\label{gdp:yeah}
 \frac{|\mu|^a(B_{r_k}(x_0))}{|\mu|^s(B_{r_k}(x_0))}\to 0, \qquad  \dashint_{B_{r_k}(x_0)} \biggl|\frac{\di \mu}{\di |\mu|}(x)-\frac{\di \mu}{\di |\mu|}(x_0)\biggr| \dd |\mu|^s(x)\to 0,
\end{equation}
we have that 
\begin{equation}\label{gdp:yeah2}
|\lambda_0\chi \nu_k-\chi \mu_k|(\R^d)\le |\lambda_0 \nu_k- \mu_k|(B_1)\to 0.
\end{equation}
Using the \(\Acal\)-freeness of \(\mu_k\) (which trivially follows from the one of \(\mu\)) we can derive an equation for \(\chi \nu_k\):
\begin{equation}\label{gdp:yeah3}
\sum_{j=1}^d A_j \lambda_0\partial_{j} (\chi \nu_k)=\sum_{j=1}^d A_j\partial_{j} (\lambda_0\chi \nu_k-\chi\mu_k)+\sum_{j=1}^d A_j \mu_k \partial_j\chi.
\end{equation}
Since we are essentially dealing with a-priori estimates, in the following we   treat measures as if they were smooth \(\Lrm^1\)-functions;   this can be achieved by a sufficiently fast regularization, see~\cite{DePhilippisRindler16} for more details.

Taking the Fourier transform of equation \eqref{gdp:yeah3} (note that we are working with compactly supported functions) we obtain
\begin{equation}\label{gdp:yeah4}
\mathbb A(\xi)\lambda_0 \widehat{\chi \nu_k} (\xi )=\mathbb A(\xi) \widehat{V_k} (\xi )+\widehat {R_k} (\xi),
\end{equation}
where
\begin{equation}\label{gdp:quasi1}
V_k:=\lambda_0\chi \nu_k-\chi\mu_k\qquad \textrm{satisfies}\qquad|V_k|(\R^d)\to 0
\end{equation}
and 
\begin{equation}\label{gdp:quasi2}
R_k:=\sum_{j=1}^d A_j \mu_k \partial_j\chi \qquad  \textrm{satisfies}\qquad \sup_k |R_k|(\R^d)\le C.
\end{equation}
Scalar multiplying   \eqref{gdp:yeah4} by  \(\overline{\mathbb A(\xi)P_0}\), adding \(\widehat{\chi \nu_k}\) to both sides and rearranging the terms, we arrive to 
\begin{equation}\label{gdp:sticazzi}
\begin{split}
 \widehat{\chi \nu_k} (\xi )&=\phantom{:} \frac{ \overline{\mathbb  A(\xi)\lambda_0}  \mathbb A(\xi) \widehat{V_k} (\xi )}{1+|\mathbb  A(\xi)\lambda_0|^2}+\frac{\overline{ \mathbb A(\xi)\lambda_0} \cdot \widehat {R_k} (\xi)}{1+|\mathbb  A(\xi)\lambda_0|^2}+\frac{\widehat {\chi\nu_k}(\xi)}{1+|\mathbb  A(\xi)\lambda_0|^2}\\
 &=: T_0(V_k)+T_1(R_k)+T_2(\chi \nu_k),
\end{split}
\end{equation}
where 
\begin{align*}
T_0[V] &= \mathcal F^{-1} \bigl[m_0(\xi)\hat {V}(\xi)\bigr],\\
T_1[R]&=\mathcal F^{-1} \bigl[m_1(\xi)(1+4\pi^2|\xi|^2)^{-1/2} \hat {R}(\xi)\bigr],\\
T_2[u]&= \mathcal F^{-1} \bigl[m_2(\xi)(1+4\pi^2|\xi|^2)^{-1}\hat {u}(\xi)\bigr],
\end{align*}
and we have set
\begin{align*}
m_0(\xi)&:=(1+|\mathbb A(\xi)\lambda_0|^2)^{-1}\overline{\mathbb  A(\xi)\lambda_0}  \mathbb A(\xi),\\
m_1(\xi)&:=(1+|\mathbb A(\xi)\lambda_0|^2)^{-1}(1+4\pi^2|\xi|^2)^{1/2}\overline{\mathbb  A(\xi)\lambda_0},\\
m_2(\xi)&:=(1+|\mathbb A(\xi)\lambda_0|^2)^{-1}(1+4\pi^2|\xi|^2).
\end{align*}
We now note that since \(\lambda_0\notin \Lambda_{\Acal}\), by homogeneity there exists \(c>0\) such that \(|\mathbb A(\xi)\lambda_0|\ge c|\xi|\) (this is the ellipticity condition we mentioned at the beginning). Hence, the symbols \(m_i\), \(i=1, 2, 3\), satisfy the assumptions of the H\"ormander--Mihlin multiplier theorem~\cite[Theorem 5.2.7]{Grafakos14book1}, i.e.\ there exists constants $K_\beta > 0$ such that
\[
|\partial^\beta m_i (\xi)|\le K_\beta |\xi|^{-|\beta|}\qquad\text{for all $\beta\in \mathbb N^d$.}
\]
This implies that  \(T_0\) is a bounded operator from  \(\Lrm^1\) to \(\Lrm^{1,\infty}\) and thus,  thanks to~\eqref{gdp:quasi1}, we get
\begin{equation}\label{gdp:piove1}
\|T_0(V_k)\|_{\Lrm^{1,\infty}} \le C|V_k|(\R^d)\to 0.
\end{equation}
Moreover,
\begin{equation}\label{gdp:piove2}
\langle T_0( V_k),\varphi\rangle =  \langle V_k, T_0^*(\varphi)\rangle \to 0 \qquad\textrm{ for every \(\varphi\in \Dcal(\R^d)\)}.
\end{equation}
where \(T_0^*\) is the adjoint operator of \(T_0\). We also observe
\[
T_1=Q_{m_1} \circ ({\rm Id}-\Delta)^{-1/2}\qquad  \textrm{and} \qquad T_2=Q_{m_2}\circ ({\rm Id}-\Delta)^{-1},
\]
where \( Q_{m_1}\) and \(Q_{m_2}\) are  the Fourier multipliers  operators associated with the symbols \(m_1\) and \(m_2\), respectively. In particular, again by the H\"ormander--Mihlin multiplier theorem, these operators are  bounded   from \(\Lrm^p\) to \(\Lrm^p\) for every \(p\in (1,\infty)\). Moreover,   \(({\rm Id}-\Delta)^{-s/2}\) is a compact operator from\footnote{Here we denote by  \(\Lrm_c^1(B_1)\) the space of \(\Lrm^1\)-functions vanishing outside \(B_1\).}  \(\Lrm_c^1(B_1)\) to   \(\Lrm^q\) for some  \(q=q(d,s)>1\),  see for instance~\cite[Lemma 2.1]{DePhilippisRindler16}.
In conclusion, by~\eqref{gdp:quasi2} and \(\sup_k |\chi \nu_k|(\R^d)\le C\),
\begin{equation}\label{gdp:piove3}
\big\{T_1(R_k)+T_2(\chi \nu_k)\big\}_{k\in \mathbb N} \qquad \textrm{is pre-compact in \(\Lrm^1(B_1)\).}
\end{equation}
Hence, combining equation~\eqref{gdp:sticazzi} with \eqref{gdp:piove1}, \eqref{gdp:piove2} and~\eqref{gdp:piove3} implies that 
\[
\chi \nu_k=u_k+w_k,
\]
where  \(u_k\to 0\) in \(\Lrm^{1,\infty}\), \(u_k\toweakstar 0\)  in the sense of distributions and \((w_k)\) is pre-compact in \(\Lrm^1(B_1)\). Since \(\chi\nu_k\ge 0\),  
\[
u_k^-:=\max\{-u_k,0\} \le |w_k|,
\]
so that the sequence \((u_k^-)\) is pre-compact in \(\Lrm^1(B_1)\). Since  \(u_k\to 0\) in \(\Lrm^{1,\infty}\), Vitali's convergence theorem implies that \(u_k^-\to 0\) in \(\Lrm^1(B_1)\) which, combined with \(u_k\toweakstar 0\), easily yields that \(u_k\to 0\) in \(\Lrm^1(B_1)\), see~\cite[Lemma 2.2]{DePhilippisRindler16}. In conclusion, \((\chi \nu_k)\) is pre-compact in \(\Lrm^1(B_1)\). Together with \eqref{gdp:totvar} and the weak* convergence of \(\mu_k\) to \(\lambda_0\nu\), this implies
\[
|\mu_k-\lambda_0\nu|(B_{1/2})\to 0,
\]
which concludes the proof.

\subsection*{Acknowledgements}

G.~D.~P.\ is supported by the MIUR SIR-grant ``Geometric Variational Problems" (RBSI14RVEZ). F.~R.\ acknowledges the support from an EPSRC Research Fellowship on ``Singularities in Nonlinear PDEs'' (EP/L018934/1).



\providecommand{\bysame}{\leavevmode\hbox to3em{\hrulefill}\thinspace}
\providecommand{\MR}{\relax\ifhmode\unskip\space\fi MR }
\providecommand{\MRhref}[2]{%
  \href{http://www.ams.org/mathscinet-getitem?mr=#1}{#2}
}
\providecommand{\href}[2]{#2}

\end{document}